\documentclass[12pt]{amsart}
\usepackage{amssymb,bbold}
\usepackage[all]{xy}

\theoremstyle{plain}
\newtheorem{theorem}{Theorem}
\newtheorem{corollary}[theorem]{Corollary}
\newtheorem{lemma}[theorem]{Lemma}
\newtheorem{proposition}[theorem]{Proposition}

\theoremstyle{definition}
\newtheorem{definition}[theorem]{Definition}
\newtheorem{remark}[theorem]{Remark}

\begin{document}
\baselineskip 18pt

\title[Unbounded continuous operators and UBSP in Banach lattices]
{Unbounded continuous operators and unbounded Banach-Saks property in Banach lattices}

\author{Omid Zabeti}

\address
  {Department of Mathematics, Faculty of Mathematics, Statistics, and Computer science,
   University of Sistan and Baluchestan, Zahedan,
   P.O. Box 98135-674. Iran}
\email{o.zabeti@gmail.com}

\maketitle
\date{\today}
\begin{abstract}
Motivated by the equivalent definition of a continuous operator between Banach spaces in terms of weakly null nets, we introduce unbounded continuous operators by replacing weak convergence with the unbounded absolutely weak convergence ( $uaw$-convergence) in the definition of a continuous operator between Banach lattices. We characterize order continuous Banach lattices and reflexive Banach lattices in terms of these spaces of operators. Moreover, motivated by characterizing of a reflexive Banach lattice in terms of unbounded absolutely weakly Cauchy  sequences, we consider pre-unbounded operators between Banach lattices which maps $uaw$-Cauchy sequences to weakly ( $uaw$- or norm) convergent sequences. This allows us to characterize $KB$-spaces and reflexive spaces in terms of these operators, too. Furthermore, we consider the unbounded Banach-Saks property as an unbounded version of the weak Banach-Saks property. There are many considerable relations between spaces possessing the unbounded Banach-Saks property with spaces fulfilled by different types of the known Banach-Saks property. In particular, we characterize order continuous Banach lattices in terms of these relations, as well.
\end{abstract}

\section{motivation and introduction}
Let us first start with some motivation. Suppose $X$ and $Y$ are Banach spaces and $T:X\to Y$ is an operator. It is known that (see \cite[Theorem 5.22]{AB} for example), $T$ is continuous if and only if it preserves weakly null nets. On the other hand, unbounded convergences have received much attention recently with some deep and inspiring results ( see \cite{DOT, KMT, GTX, Z}). In particular, unbounded absolute weak convergence ( $uaw$-convergence, for short) as a weak version of the unbounded norm convergence and also as an unbounded version of the weak convergence has been investigated in \cite{Z}. Furthermore, unbounded absolute weak Dunford-Pettis operators ( $uaw$-Dunford-Pettis operators, in brief), as an unbounded version of Dunford-Pettis operators, have been considered recently in \cite{EGZ}. In particular, it is shown in \cite{Z} that unbounded absolute weak convergence enables us to characterize order continuous Banach lattices and also reflexive Banach lattices, thoroughly. So, it is of independent interest to consider continuous operators that utilize $uaw$-convergence to see whether or not we can characterize order continuity or reflexivity of the underlying space in terms of these classes of operators. Furthermore, it would also be worthwhile to investigate some possible relations between these classes of operators with the well-known classes of operators such as $M$-weakly compact operators, $L$-weakly compact operators and weakly compact ones; this is partially done for Dunford-Pettis operators in \cite{EGZ}.
Again, let us consider some motivation. It is proved in \cite[Theorem 8]{Z} that, a Banach lattice $E$ is reflexive if and only if every bounded $uaw$-Cauchy sequence in $E$ is weakly convergent; this, in turn, implies that the identity operator on $E$ possesses this property: it maps every bounded $uaw$-Cauchy sequence to a weakly convergent sequence. This motivates us to define operators which respect this property.
So, it would be interesting to consider pre-unbounded continuous operators by replacing $uaw$-null sequences with $uaw$-Cauchy sequences. This enables us to characterize $KB$-spaces and reflexive Banach lattices in terms of these classes of operators, too.

Finally, let us consider the last motivation. It is known that a Banach space $X$ possesses the weak Banach-Saks property if every weakly null sequence in $X$ has a subsequence whose Ces\`{a}ro means is convergent. Further to the previous attitude, it would be  inspiring to replace weak convergence in this definition by $uaw$-convergence. Moreover, it is also interesting to investigate possible relations of the unbounded version of the Banach-Saks property with the other known cases of this property.

We organize the paper as follows.
In the second section, we investigate unbounded continuous operators by replacing weak convergence with the $uaw$-convergence in the definition of a continuous operator and characterize order continuity in the dual space and also reflexive spaces in terms of these classes of operators. In the third section, we investigate pre-unbounded continuous operators by replacing $uaw$-null sequences with $uaw$-Cauchy sequences; we study their relations with the unbounded continuous operators. Furthermore, we characterize $KB$-spaces and reflexive Banach lattices in terms of them, too. In the fourth section, we consider the unbounded Banach-Saks property by replacing weak convergence in the definition of the weak Banach-saks property by $uaw$-convergence. We investigate its relations with other important types of the Banach-Saks property; namely, the Banach-saks property, the weak Banach-Saks property, the disjoint Banach-Saks property and the disjoint weak Banach-Saks property. There are, again, many interesting characterizations of order continuous Banach lattices in terms of the relations between the unbounded Banach-Saks property and the known kinds of this property.

Before we proceed more, let us consider some preliminaries.

Suppose $E$ is a Banach lattice. A net $(x_{\alpha})$ in $E$ is said to be {\bf unbounded absolute weak convergent} ($uaw$-convergent, for short) to $x\in E$ if for each $u\in E_{+}$, $|x_{\alpha}-x|\wedge u\rightarrow0$ weakly, in notation $x_{\alpha}\xrightarrow{uaw}x$. $(x_{\alpha})$ is {\bf unbounded norm convergent} ($un$-convergent, in brief) if $\||x_{\alpha}-x|\wedge u\|\rightarrow 0$, in notation $x_{\alpha}\xrightarrow{un}x$. Both convergences are topological. For ample information on these concepts, see \cite{DOT, KMT, Z}.

For undefined terminology and concepts, we refer the reader to \cite{AB}. All operators in this note, are assumed to be continuous, unless otherwise stated, explicitly.
\section{unbounded continuous operators}
First, we define unbounded continuous operators between Banach lattices; the third part is defined initially in \cite{EGZ}.
\begin{definition}
Suppose $E,F$ are Banach lattices and $X$ is a Banach space. Then, we have the following notions.
\begin{itemize}
\item[\em (i)] {A continuous operator $T:E\to X$ is said to be unbounded continuous if for each bounded $uaw$-null sequence $(x_n)\subseteq E$, $(T(x_n))$ is weakly null}.
\item[\em (ii)]{A continuous operator $T:E\to F$ is called $uaw$-continuous if for each bounded $uaw$-null sequence $(x_n)\subseteq E$, $(T(x_n))$ is also $uaw$-null}.
\item[\em (iii)] {An operator $T:E\to X$ is referred to as $uaw$-Dunford-Pettis  if for each bounded $uaw$-null sequence $(x_n)\subseteq E$, $(T(x_n))$ is norm null}.

\end{itemize}
\end{definition}
Now, we investigate some partial relations.
\begin{proposition}\label{135}
Suppose $E$ and $F$ are Banach lattices such that $F'$ is order continuous. Then every $uaw$-continuous operator $T:E\to F$ is unbounded continuous.
\end{proposition}
\begin{proof}
Suppose $T$ is $uaw$-continuous. For every norm bounded $uaw$-null sequence $(x_n)$ in $E$, $T(x_n)\xrightarrow{uaw}0$. By \cite[Theorem 7]{Z}, $T(x_n)\xrightarrow{w}0$, as claimed.
\end{proof}
Since the lattice operations in an $AM$-space are weakly sequentially continuous, we have the following.
\begin{corollary}\label{134}
Suppose $E$ is a Banach lattice and $F$ is an $AM$-space. Then an operator $T:E\to F$ is unbounded continuous if and only if it is $uaw$-continuous.
\end{corollary}
For the converse, we have the following.
\begin{proposition}
Suppose $E$ and $F$ are Banach lattices and $T:E\to F$ is a positive operator. If $T$ is unbounded continuous, then it is $uaw$-continuous.
\end{proposition}
\begin{proof}
Suppose $(x_n)$ is a bounded positive $uaw$-null sequence in $E$. By the assumption, $T(x_n)\xrightarrow{w}0$. This implies that $T(x_n)\rightarrow0$ in the absolute weak topology. Therefore, $T(x_n)\xrightarrow{uaw}0$.
\end{proof}
\begin{corollary}
Suppose $E$ and $F$ are Banach lattices such that $F'$ possesses an order continuous norm. Then, a positive operator $T:E\to F$ is unbounded continuous if and only if it is $uaw$-continuous.
\end{corollary}
The following lemma is an easy application of \cite[Theorem 7]{Z}.
\begin{lemma}\label{8}
Suppose $E$ is a Banach lattice. Then $E'$ is order continuous if and only if the identity operator $I$ on $E$ is unbounded continuous.
\end{lemma}
Furthermore, the following is immediate.
\begin{corollary}\label{133}
Suppose $E$ is a $KB$-space. Then the identity operator $I$ on $E$ is unbounded continuous if and only if $E$ is reflexive.
\end{corollary}

Observe that $KB$-space assumption is necessary in Corollary \ref{133} and can not be omitted. The identity operator $I$ on $c_0$ is unbounded continuous but $c_0$ is not reflexive.

In this part, we provide some ideal properties.
\begin{proposition}\label{7}
 Let $E,F,G$ be Banach lattices and $X,Y$ be Banach spaces. Then the following assertions hold.
	\begin{itemize}
		\item[\em (i)] {If $T:X\to Y$ is Dunford-Pettis and $S:E\to X$ is unbounded continuous then $TS$ is $uaw$-Dunford-Pettis}.
		\item[\em (ii)] {If $T:F\to X$ is a $uaw$-Dunford--Pettis operator and $S:E\to F$ is $uaw$-continuous, then $TS$ is $uaw$-Dunford-Pettis}.
		\item[\em (iii)]
	{ If $T:X\to Y$ is continuous and $S:E\to X$ is an unbounded continuous operator, then $TS$ is also unbounded continuous}.
		\item[\em (iv)]{ If $T:F\to G$ is an onto lattice homomorphism  and $S:E\to F$ is $uaw$-continuous, then $TS$ is also $uaw$-continuous}.
	\end{itemize}
\end{proposition}
\begin{proof}
$(i)$. Suppose $(x_n)$ is a norm bounded $uaw$-null sequence in $E$. So, $S(x_n)\xrightarrow{w}0$. Therefore, $\|TS(x_n)\|\rightarrow 0$.

$(ii)$. Suppose $(x_n)$ is a norm bounded $uaw$-null sequence in $E$. Therefore, $S(x_n)\xrightarrow{uaw}0$. Thus, $\|TS(x_n)\|\rightarrow 0$.

$(iii)$. Suppose $(x_n)$ is a norm bounded $uaw$-null sequence in $E$. By the assumption, $S(x_n)\xrightarrow{w}0$ so that $TS(x_n)\xrightarrow{w}0$.

$(iv)$. First, observe that for each $u\in G_{+}$, there is a $v\in F_{+}$ with $T(v)=u$. Suppose $(x_n)$ is a norm bounded $uaw$-null sequence in $E$.
By the assumption, $S(x_n)\xrightarrow{uaw}0$ so that $|S(x_n)|\wedge v\xrightarrow{w}0$. Therefore,
\[|TS(x_n)|\wedge u= T(|S(x_n|)\wedge u=T(|S(x_n)|\wedge v)\xrightarrow{w}0.\]
\end{proof}
Furthermore, we see that unbounded continuous operators are very closed to continuous operators. The following lemma is an easy consequence of \cite[Theorem 7]{Z}.
\begin{lemma}\label{300}
Suppose $E$ is a Banach lattice whose dual space is order continuous and $X$ is any Banach space. Then every continuous operator $T:E\to X$ is unbounded continuous.
\end{lemma}
\begin{theorem}\label{402}
For a Banach lattice $E$, the following are equivalent.
\begin{itemize}
		\item[\em (i)] {$E'$ is order continuous}.
		\item[\em (ii)] {For any Banach space $X$, every continuous operator $T:E\to X$ is unbounded continuous}.
		\end{itemize}
\end{theorem}
\begin{proof}
$(i)\Rightarrow(ii)$. This is a consequence of Lemma \ref{300}.

$(ii)\Rightarrow(i)$. Take $X=E$ and $T=I$; the identity operator on $E$; by the assumption, $I$ is unbounded continuous and then use Lemma \ref{8}.
\end{proof}
Now, we can characterize reflexive Banach lattices in terms of unbounded continuous operators.
\begin{corollary}\label{401}
For a $KB$-space $E$, the following are equivalent.
\begin{itemize}
		\item[\em (i)] {$E$ is reflexive}.
		\item[\em (ii)] {For any Banach space $X$, every continuous operator $T:E\to X$ is unbounded continuous}.
		\end{itemize}
\end{corollary}
Suppose $E$ and $F$ are Banach lattice. The class of all unbounded continuous operators from $E$ into $F$ is denoted by $B_{uc}(E,F)$, the class of all $uaw$-continuous operators between them will get the terminology $B_{uaw}(E,F)$. In this step, we consider some closedness properties for these classes of continuous operators.
\begin{proposition}
Suppose $E$ and $F$ are Banach lattice. Then $B_{uc}(E,F)$ is closed as a subspace of the space of all continuous operators from $E$ into $F$.
\end{proposition}
\begin{proof}
Suppose $(T_m)$ is a sequence of unbounded continuous operators which is convergent to the operator $T$. We need to show that $T$ is unbounded continuous. For any $\varepsilon>0$, there is an $m_0$ such that $\|T_{m_0}-T\|<\frac{\varepsilon}{2}$. So, for each $x$ with $\|x\|\leq 1$, $\|T_{m_0}(x)-T(x)\|<\frac{\varepsilon}{2}$. Assume that $(x_n)$ is a norm bounded $uaw$-null sequence in $E$. This means that $\|T_{m_0}(x_n)-T(x_n)\|<\frac{\varepsilon}{2}$. Note that $T_{m_0}(x_n)\xrightarrow{w}0$ so that for a fixed $f\in F^{'}_{+}$ and sufficiently large $n$, $ f(T_{m_0}(x_n))<\frac{\varepsilon}{2}$ so that $ f(T(x_n))<\varepsilon$.
\end{proof}
\begin{proposition}
Suppose $E$ and $F$ are Banach lattice. Then $B_{uaw}(E,F)$ is closed as a subspace of the space of all continuous operators from $E$ into $F$.
\end{proposition}
\begin{proof}
Suppose $(T_m)$ is a sequence of $uaw$-continuous operators which is convergent to the operator $T$. We need to show that $T$ is $uaw$-continuous. For any $\varepsilon>0$, there is an $m_0$ such that $\|T_{m_0}-T\|<\frac{\varepsilon}{2}$. So, for each $x$ with $\|x\|\leq 1$, $\|T_{m_0}(x)-T(x)\|<\frac{\varepsilon}{2}$. Assume that $(x_n)$ is a norm bounded $uaw$-null sequence in $E$. This means that $\|T_{m_0}(x_n)-T(x_n)\|<\frac{\varepsilon}{2}$. Note that $T_{m_0}(x_n)\xrightarrow{uaw}0$. Fix $f\in F^{'}_{+}$ and $u\in F_{+}$. Observe that for $a,b,c\geq 0$ in an Archimedean vector lattice, we have $|a\wedge c-b\wedge c|\leq |a-b|\wedge c$. Therefore,
\[f(|T_{m_0}(x_n)|\wedge u)-f(|T(x_n)|\wedge u)\leq \||(T_{m_0}((x_n))|\wedge u)-(|T(x_n)|\wedge u)\|\]
\[\leq \|||T_{m_0}(x_n)|-|T(x_n)||\wedge u\|\leq \|||T_{m_0}(x_n)|-|T(x_n)||\|\leq \|T_{m_0}(x_n)-T(x_n)\|<\frac{\varepsilon}{2}.\]

 On the other hand, for sufficiently large $n$, $ f(T_{m_0}(x_n)\wedge u)\rightarrow 0$. This, in turn, results in $f(T(x_n)\wedge u)\rightarrow 0$.
\end{proof}
\begin{remark}
Neither $B_{uc}(E,F)$ nor $B_{uaw}(E,F)$ are order closed, in general. Consider operator $S:\ell_1 \to L_2[0,1]$ defined via $S(\alpha_1,\alpha_2,\ldots)=\Sigma_{n=1}^{\infty}\alpha_n {r_n}^{+}$. In which $(r_n)$ denotes the sequence of Rademacher functions; for more details about this operator consider \cite[Example 5.17, page 284]{AB}. Note that $S$ is not unbounded continuous; the standard basis $(e_n)$ is $uaw$-null in $\ell_1$ but $S(e_n)={r_n}^{+}$ which is not weakly null since $\int_{0}^{1}{r_n}^{+}(t)dt =\frac{1}{2}$. Put $S_n=SP_n$, in which $P_n$ is the canonical projection on $\ell_1$. Observe that each $S_n$ is finite rank. So, in the range space, $uaw$-convergence, weak convergence and norm one agree.
Therefore, we conclude that each $S_n$ is both $uaw$-continuous and unbounded continuous. Moreover $S_n\uparrow S$ but $S$ is neither unbounded continuous nor $uaw$-continuous.
\end{remark}
\begin{remark}
Suppose $E$ and $F$ are Banach lattices and $T,S:E\to F$ are continuous operators such that $0\leq S\leq T$. It can be easily verified that if $T$ is either unbounded continuous or $uaw$-continuous, then so is $S$.
\end{remark}
In this part, we extend Theorem \ref{402} while $(ii)\Rightarrow (i)$ is improved.
\begin{theorem}\label{2021}
Suppose $E$ is a Banach lattice. For every operator $T:E\to \ell_{\infty}$, the following assertions are equivalent.
\begin{itemize}
		\item[\em (i)] {$E'$ is order continuous}.
	\item[\em (ii)] {Every continuous operator $T:E\to \ell_{\infty}$ is unbounded continuous}.
		\end{itemize}
\end{theorem}
\begin{proof}
$(i)\Rightarrow (ii)$. It is a direct consequence of Theorem \ref{402}.

$(ii)\Rightarrow (i)$. Suppose not. So, by \cite[Theorem 4.69]{AB} there exists a positive bounded disjoint sequence $(x_n)\subseteq E$ which is not weakly null.
By \cite[Lemma 3.4]{AEH}, there exists a bounded positive disjoint sequence $(g_n)$ in $E'$  such that $g_n(x_n)=1$ and $g_n(x_m)=0$ for all $m \neq n$. Define the operator $T:E\to \ell_{\infty}$ by $T(x)=(g_1(x),(g_1+g_2)(x),\ldots)$. $T$ is well-defined because the sequence $(g_i)$ is disjoint and bounded. Observe that $x_n\xrightarrow{uaw}0$ by \cite[Lemma 2]{Z} but $T(x_n)=(0,\ldots,0,1,1,\ldots)$ in which the zero is appeared $n$-times. By using Dini's theorem (\cite[Theorem 3.52]{AB}), $T(x_n)\nrightarrow 0$ weakly, a contradiction.
\end{proof}
Moreover, by considering Corollary \ref{134}, we obtain the following.
\begin{corollary}
Suppose $E$ is a Banach lattice. For every operator $T:E\to \ell_{\infty}$, the following assertions are equivalent.
\begin{itemize}
		\item[\em (i)] {$E'$ is order continuous}.
	\item[\em (ii)] {Every continuous operator $T:E\to \ell_{\infty}$ is $uaw$-continuous}.
		\end{itemize}
\end{corollary}
In the following, we prove one of the main results in \cite{EGZ} under less hypotheses; more precisely, we show that $uaw$-Dunford-Pettis operators and $M$-weakly compact operators are in fact the same.

Recall that an operator $T:E\to X$, where $E$ is a Banach lattice and $X$ is a Banach space, is called $M$-weakly compact if for every norm bounded disjoint sequence $(x_n)$ in $E$, $\|T(x_n)\|\rightarrow 0$. Moreover, $T$ is said to be $o$-weakly compact if $T[0,x]$ is a weakly relatively compact set for every $x\in E_{+}$. Also, note that a continuous operator $T:X\to E$ is said to be $L$-weakly compact if every disjoint sequence in the solid hull of $T(B_X)$ is norm null.
\begin{theorem}\label{400}
Suppose $E$ is a Banach lattice and $X$ is a Banach space. If $T:E\to X$ is an $M$-weakly compact operator, then it is $uaw$-Dunford-Pettis.
\end{theorem}
\begin{proof}
Suppose $(x_n)$ is a positive norm bounded $uaw$-null sequence in $E$. This means that $x_n\wedge u\xrightarrow{w}0$ for any positive $u\in E$. Observe that $T$ is also $o$-weakly compact by \cite[Theorem 5.57]{AB}. So, using \cite[Exercise 3, Page 336]{AB}, convinces us that $\|T(x_n\wedge u)\|\rightarrow 0$. Note that by \cite[Theorem 5.60]{AB}, the proof would be complete.
\end{proof}
By using Theorem \ref{400} and \cite[Proposition 2.6]{EGZ}, we conclude that the notions of an $M$-weakly compact operator and a $uaw$-Dunford-Pettis operator agree. Therefore, many results regarding $uaw$-Dunford-Pettis operators mentioned in \cite{EGZ} can be restated directly for $M$-weakly compact operators.
Finally, we shall show that unbounded continuous operators are very closed to $M$-weakly compact operators and $uaw$-continuous operators are related to $L$-weakly compact operators. Before we proceed, we consider the following definition.

Suppose $E$ is a Banach lattice and $X$ is a Banach space. A continuous operator $T:E\to X$ is said to be  $WM$-weakly compact if for every bounded disjoint sequence $(x_n)\subseteq E$, $T(x_n)\xrightarrow{w}0$. A continuous operator $T:X\to E$ is called a $WL$-weakly compact operator if $y_n\xrightarrow{w}0$ for each disjoint sequence $(y_n)$ in the solid hull of $T(B_X)$. It is easy to see that every $M$-weakly compact operator is $WM$-weakly compact and every $L$-weakly compact operator is a $WL$-weakly compact operator. Nevertheless, the inclusion map from $c_0$ into $\ell_{\infty}$ is both $WM$-weakly compact  and $WL$-weakly compact, although, it is neither an $M$-weakly compact operator nor an $L$-weakly compact operator.

These classes of continuous operators enjoy an approximation property similar to \cite[Theroem 5.60]{AB}.
\begin{lemma}\label{500}
For Banach lattices $E$ and $F$, and a Banach space $X$, the following assertions hold.
\begin{itemize}
\item[\em (i)] {If $T:E\to F$ is a $WM$-weakly compact operator, then for each $\varepsilon>0$ and for each $f\in {F'}_{+}$, there exists some $u\in E_{+}$ such that $f(T(|x|-u)^{+})<\varepsilon$ satisfies for all $x\in E$ with $\|x\|\leq 1$}.
		\item[\em (ii)] {If $T:X\to E$ is a $WL$-weakly compact operator, then for each $\varepsilon>0$ and for each $f\in {E'}_{+}$, there exists some $u\in E_{+}$ such that $f((|Tx|-u)^{+})<\varepsilon$ satisfies for all $x\in X$ with $\|x\|\leq 1$}.
\end{itemize}
\end{lemma}
\begin{proof}
$(i)$. Suppose $B$ is the closed unit ball of $E$ and $(x_n)$ is a disjoint sequence in $B$. Given $\varepsilon>0$ and $f\in {F'}_{+}$. By the assumption, $f(Tx_n)\rightarrow 0$. By \cite[Theorem 4.36]{AB} ( with $\rho(x)=f(|x|)$), there exists $u\in E_{+}$ with $f(T(|x|-u)^{+})<\varepsilon$ for all $x$ with $\|x\|\leq 1$.

$(ii)$. Suppose $U$ is the closed unit ball of $X$ and $A$ is the solid hull of $T(U)$. $T$ is $WL$-weakly compact so that every disjoint sequence in $A$ is weakly null. Consider the identity operator $I$ on $E$ and fix $f\in {E'}_{+}$. Again, by considering $\rho(x)=f(|x|)$ and using \cite[Theorem 4.36]{AB}, we
can find $u\in E_{+}$ such that $f(I(|y|-u)^{+})<\varepsilon$ for each $y\in A$ so that $f((|Tx|-u)^{+})<\varepsilon$ for all $x\in U$.
\end{proof}
Therefore, we have the following.

\begin{theorem}
Suppose $E$ and $F$ are Banach lattices. Then a continuous operator $T:E\to F$ is unbounded continuous if and only if it is a $WM$-weakly compact operator.
\end{theorem}
\begin{proof}
The direct implication is a consequence of  \cite[Lemma 2]{Z}. For the other side, assume that $T:E\to F$ is $WM$-weakly compact and $(x_n)$ is a bounded $uaw$-null sequence in $E$. Given $\varepsilon>0$ and $f\in {F'}_{+}$. Considering Lemma \ref{500}, we can find $u\in E_{+}$ such that $f(T(|x_n|-u)^{+})<\varepsilon$ for all $n\in \Bbb N$. Note that $x_n\wedge u\xrightarrow{w}0$ so that $T(x_n\wedge u)\xrightarrow{w}0$. Therefore, $T(x_n)\xrightarrow{w}0$, as claimed.
\end{proof}
Now, we state a version of \cite[Theorem 2.10]{EGZ} for the case of unbounded continuous operators.
\begin{proposition}
Suppose $E$ and $F$ are Banach lattices. Then every $WL$-weakly compact and  $uaw$-continuous operator $T:E\to F$ is unbounded continuous.
\end{proposition}
\begin{proof}
Suppose $(x_n)$ is a bounded $uaw$-null sequence in $E$. Given $\varepsilon>0$ and $f\in {F'}_{+}$. By Lemma \ref{500}, there exists $u\in F_{+}$ such that $f((|T(x_n)|-u)^{+})<\varepsilon$. Note that $|T(x_n)|\wedge u\xrightarrow{w}0$. Thus, $T(x_n)\xrightarrow{w}0$, as desired.
\end{proof}
By using \cite[Theorem 4.34]{AB}, we obtain the following surprising but simple result.
\begin{proposition}
Suppose $X$ is a Banach space and $E$ is a Banach lattice. Then every weakly compact operator $T:X\to E$ is $WL$-weakly compact.
\end{proposition}
\begin{remark}
Suppose $E$ and $F$ are Banach lattices. If $F$ has a strong unit, then notions of boundedness and order boundedness in $F$ agree. So, it is easy to see that every continuous operator $T:E\to F$ is $WL$-weakly compact. Moreover, when $E$ has a strong unit, one may verify that every continuous operator $T:E\to F$ is $WM$-weakly compact.
\end{remark}
\section{pre-unbounded continuous operators}

\begin{definition}
Suppose $E,F$ are Banach lattices and $X$ is a Banach space. Then we have the following concepts.
\begin{itemize}
\item[\em (i)] {A continuous operator $T:E\to X$ is called pre-unbounded continuous if for each bounded $uaw$-Cauchy sequence $(x_n)\subseteq E$, $(T(x_n))$ is weakly convergent}.
\item[\em (ii)]{A continuous operator $T:E\to F$ is said to be pre-$uaw$-continuous if for each bounded $uaw$-Cauchy sequence $(x_n)\subseteq E$, $(T(x_n))$ is $uaw$-convergent}.
\item[\em (iii)] {A continuous operator $T:E\to X$ is referred to as pre-$uaw$-Dunford-Pettis  if for each bounded $uaw$-Cauchy sequence $(x_n)\subseteq E$, $(T(x_n))$ is norm convergent}.

\end{itemize}
\end{definition}

These operators are abundant in the category of all continuous operators. We shall show that when $E$ is a $KB$-space, every unbounded continuous operator on $E$ is pre-unbounded continuous; or every $uaw$-Dunford-Pettis operator ( $M$-weakly compact operator) is pre-$uaw$-Dunford-Pettis. Moreover, when $E$ is reflexive, every continuous operator on $E$ is pre-unbounded continuous.

First of all, we consider the following; it states that in $KB$-spaces, we have an obvious relation between different kinds of unbounded continuous  operators and the corresponding types of pre-unbounded continuous operators. It is an easy combination of \cite[Theorem 4]{Z} and \cite[Theorem 4.6]{KMT}.
\begin{proposition}\label{101}
Suppose $E$ is a $KB$-space, $F$ is a  Banach lattice, and $X$ is a Banach space. Then we have the following observations.
\begin{itemize}
		\item[\em (i)] {Every unbounded continuous operator $T:E\to X$ is pre-unbounded continuous}.
		\item[\em (ii)] {Every $uaw$-continuous operator $T:E\to F$ is pre-$uaw$-continuous}.
\item[\em (iii)] {Every $uaw$-Dunford-Pettis operator $T:E\to X$ is pre-$uaw$-Dunford-Pettis}.
		\end{itemize}
\end{proposition}
\begin{remark}
Note that being $KB$-space is essential in Proposition \ref{101} and can not be removed. Consider the identity operator $I$ on $c_0$; it is unbounded continuous and also $uaw$-continuous. Nevertheless, by considering the sequence $(u_n)$ defined via $u_n=(1,\ldots,1,0,\ldots)$, we see that $I$ is neither pre-unbounded continuous nor pre-$uaw$-continuous.
Moreover, consider the operator $T:c_0\to \ell_1$ defined via $T((x_n))=(\frac{x_n}{n})$. $T$ is $uaw$-Dunford-Pettis; for if $(x_n)\subseteq c_0$ is norm bounded and $uaw$-null, then it is weakly null by \cite[Theorem 7]{Z} so that $T((x_n))$ is weakly null in $\ell_1$ and therefore norm null by the Schur property. But, it fails to be pre-$uaw$-Dunford-Pettis; again, by considering the sequence $(u_n)$ as before.
\end{remark}
To obtain some results for the other direction, we need the following useful fact. It is an extension of \cite[Lemma 9.10]{KMT}.
\begin{lemma}\label{704}
Suppose $E$ is a Banach lattice and $(x_n)\subseteq E$ is a sequence such that $x_n\xrightarrow{uaw}x$ and $x_n\xrightarrow{w}y$. Then $x=y$.
\end{lemma}
\begin{proof}
We may assume that $y=0$. By \cite[Theorem 4.37]{AB}, for each $\varepsilon>0$ and for each $f\in {E'}_{+}$, there exists some $u\in E_{+}$ such that $f(|x_n-x|-|x_n-x|\wedge u)<\varepsilon$, for each $n\in \Bbb N$. But for sufficiently large $n$, $|x_n-x|\wedge u\xrightarrow{w}0$ so that $x_n\xrightarrow{w}x$; this means that $x=0$.
\end{proof}
Now, we have the following.
\begin{proposition}
Suppose $E$ and $F$ are  Banach lattices. Then we have the following observations.
\begin{itemize}
		\item[\em (i)] {Every $uaw$-continuous operator $T:E\to F$ which is also pre-unbounded continuous, is unbounded continuous}.
		\item[\em (ii)] {Every unbounded continuous operator $T:E\to F$ which is also pre-$uaw$-continuous, is  $uaw$-continuous.}
\item[\em (iii)] {Suppose $T:E\to F$ is pre-$uaw$-Dunford-Pettis  and also so is either  unbounded continuous or $uaw$-continuous. Then $T$ is $uaw$-Dunford-Pettis}.
		\end{itemize}
\end{proposition}
\begin{proof}
$(i)$. Suppose $(x_n)\subseteq E$ is $uaw$-null. By the assumption, $T(x_n)\xrightarrow{uaw}0$. Moreover, $T(x_n)\xrightarrow{w}x$ for some $x\in F$. By Lemma \ref{704}, $x=0$.

$(ii)$. Suppose $(x_n)\subseteq E$ is $uaw$-null. By the assumption, $T(x_n)\xrightarrow{w}0$. Moreover, $T(x_n)\xrightarrow{uaw}x$ for some $x\in F$. By Lemma \ref{704}, $x=0$.

$(iii)$. Suppose $(x_n)\subseteq E$ is $uaw$-null. By the assumption, $T(x_n)\rightarrow x$  for some $x\in F$ so that $T(x_n)\xrightarrow{w}x$ and $T(x_n)\xrightarrow{uaw}x$. Now, if $T$ is either  unbounded continuous or  $uaw$-continuous, we conclude that $x=0$.
\end{proof}
Moreover, we have the following simple ideal properties.
\begin{proposition}\label{100}
 Let $E,F,G$ be  Banach lattices and $X,Y$ be Banach spaces. Then, we have the following observations.
	\begin{itemize}
		\item[\em (i)] {If $T:E\to X$ is pre-unbounded continuous and $S:X\to Y$ is continuous, then  $ST$ is also pre-unbounded continuous}.
		\item[\em (ii)] {If $T:E\to F$ is pre-$uaw$-continuous and $S:F\to G$ is  $uaw$-continuous, then $ST$ is also  pre-$uaw$-continuous}.
\item[\em (iii)] {If $T:E\to F$ is pre-$uaw$-continuous and $S:F\to X$ is unbounded continuous, then $ST$ is pre-unbounded continuous}.

	\end{itemize}
\end{proposition}

\begin{theorem}\label{22}
For a Banach lattice $E$, the following are equivalent.
\begin{itemize}
		\item[\em (i)] {$E$ is reflexive}.
		\item[\em (ii)] {For every Banach space $X$, every continuous operator $T:E\to X$ is pre-unbounded continuous}.
		\end{itemize}
\end{theorem}
\begin{proof}
$(i)\Rightarrow (ii)$. Suppose $E$ is reflexive and $T:E\to X$ is a continuous operator. Assume that $(x_n)$ is a bounded $uaw$-Cauchy sequence in $E$. So, by \cite[Theorem 8]{Z}, it is weakly convergent to $x\in E$. This means that the identity operator $I$ on $E$ is pre-unbounded continuous. Note that $T=TI$ and use Proposition \ref{100}.

$(ii)\Rightarrow(i)$. Put $X=E$ and $T=I$; the identity operator on $E$. This means that $I$ is pre-unbounded continuous. Again, \cite[Theorem 8]{Z}, yields the result.
\end{proof}
\begin{proposition}
For a Banach lattice $E$, the following are equivalent.
\begin{itemize}
		\item[\em (i)] {$E$ is a $KB$-space}.
		\item[\em (ii)] {For every Banach lattice $F$, every $uaw$-continuous operator $T:E\to F$ is  pre-$uaw$-continuous}.
		\end{itemize}
\end{proposition}
\begin{proof}
$(i)\Rightarrow (ii)$. It is proved in Proposition \ref{101}.

$(ii)\Rightarrow(i)$. Put $F=E$ and $T=I$; the identity operator on $E$. This means that $I$ is  pre-$uaw$-continuous. Again, \cite[Theorem 4]{Z} and \cite[Theorem 4.6]{KMT} would complete the proof.
\end{proof}
The following results are similar to  Proposition \ref{135} and Corollary \ref{134}.
\begin{proposition}
Suppose $E$ and $F$ are Banach lattices such that $F'$ is order continuous. Then every pre-$uaw$-continuous operator $T:E\to F$ is pre-unbounded continuous.
\end{proposition}

\begin{corollary}\label{403}
Suppose $E$ is a Banach lattice and $F$ is an $AM$-space. Then  an operator $T:E\to F$ is pre-unbounded continuous if and only if it is pre-$uaw$-continuous.
\end{corollary}
Now, we state a version of Theorem \ref{22} in which the part $(ii)\Rightarrow(i)$ is improved.
\begin{theorem}\label{21}
Suppose $E$ is a Banach lattice. Then the following are equivalent.
\begin{itemize}
		\item[\em (i)] {$E$ is reflexive}.
		\item[\em (ii)] {Every continuous operator $T:E\to c_0$ is pre-unbounded continuous}.
		\end{itemize}

\end{theorem}
\begin{proof}
$(i)\Rightarrow(ii)$. It is done by Theorem \ref{22}.

$(ii)\Rightarrow (i)$. Suppose not; so, by \cite[Theorem 2.4.15]{Ni}, $E$ contains a lattice copy of either $\ell_1$ or $c_0$. Moreover, by \cite[Proposition 2.3.11]{Ni}, there exists a positive projection $P:E\to \ell_1$. Therefore, the restriction of $P$ to $\ell_1$ is the identity operator on $\ell_1$ which is not pre-unbounded continuous. Now, suppose $(e_n)$ is the standard basis of $\ell_1$ which is certainly disjoint in $E$ so that $uaw$-null by \cite[Lemma 2]{Z} but it can be easily seen that it is not weakly convergent in $\ell_1$. Furthermore, by \cite[Theorem 2.4.12]{Ni}, there exists a positive projection $P:E\to c_0$. The restriction of $P$ to $c_0$ is the identity operator which is not pre-unbounded continuous; use the $uaw$-Cauchy sequence $(u_n)$ defined via $u_n=(1,\ldots,1,0,\ldots)$. Note that $(u_n)$ is in fact absolutely weakly Cauchy in $c_0$ so that absolutely weakly Cauchy in $E$. Thus, $(u_n)$ is $uaw$-Cauchy in $E$. But it is obvious that it is not certainly weakly convergent.
\end{proof}
Combining Theorem \ref{21} and Corollary \ref{403}, we have the following.
\begin{corollary}
Suppose $E$ is a Banach lattice. Then the following are equivalent.
\begin{itemize}
		\item[\em (i)] {$E$ is reflexive}.
		\item[\em (ii)] {Every continuous operator $T:E\to c_0$ is pre-$uaw$-continuous}.
		\end{itemize}

\end{corollary}
Observe that every pre-$uaw$-Dunford-Pettis operator is pre-unbounded continuous. Therefore, we have the following.
\begin{corollary}\label{23}
Suppose $E$ is a Banach lattice. If every continuous operator $T:E\to c_0$ is pre-$uaw$-Dunford-Pettis, then $E$ is reflexive.
\end{corollary}
Note that the converse of Corollary \ref{23} is not true, in general. The inclusion map $\imath:\ell_2\to c_0$ is not pre-$uaw$-Dunford-Pettis operator.
\section{unbounded Banach-Saks property}
Suppose $E$ is a Banach lattice. $E$ is said to have the {\bf unbounded Banach-Saks property} ( {\bf UBSP}, for short) if for every norm bounded $uaw$-null sequence $(x_n)\subseteq E$, there is a subsequence $(x_{n_k})$  whose Ces\`{a}ro means is convergent. Moreover, recall that $E$ possesses the disjoint Banach-Saks property ( {\bf DBSP}, for short) if every bounded disjoint sequence in $E$ has a Ces\`{a}ro convergent subsequence; $E$ has the disjoint weak Banach-Saks property ( {\bf DWBSP}, in brief) if every disjoint weakly null sequence in $E$ has a Ces\'{a}ro convergent subsequence. Furthermore, $E$ possesses the weak Banach-Saks property ( {\bf WBSP}, in brief) if for every weakly null sequence $(x_n)$, it has a subsequence which is Ces\`{a}ro convergent. Finally observe that $E$ possesses the Banach-Saks property ( {\bf BSP}) if every bounded sequence $(x_n)\subseteq E$ has a Ces\'{a}ro convergent subsequence. For a brief but comprehensive context in this subject, see \cite{GTX}.
In the following lemma, we collect some preliminary facts.
\begin{lemma}\label{1}
\begin{itemize}
\item[\em (i)] {Suppose $E'$ is order continuous. If $E$ possesses {\bf WBSP}, then it satisfies {\bf UBSP}}.
\item[\em (ii)]{{\bf UBSP} implies {\bf DBSP}}.
\item[\em (iii)] {Suppose $E$ is either an $AM$-space or an atomic order continuous Banach lattice. Then {\bf UBSP} implies {\bf WBSP}}.
\end{itemize}
\end{lemma}
\begin{proof}
$(i)$. Consider \cite[Theorem
7]{Z} and the proof would be complete.

$(ii)$. Consider \cite[Lemma 2]{Z} which asserts that every disjoint sequence is $uaw$-null.

$(iii)$. In both cases, it can be verified easily that the lattice operations are weakly sequentially continuous so that weak convergence implies $uaw$-convergence.
\end{proof}
\begin{remark}
Observe that order continuity assumption in Lemma \ref{1}$(i)$ is essential and can not be removed. For example $\ell_1$ possesses {\bf WBSP} but it fails {\bf UBSP}; consider the standard basis sequence $(e_n)$ which is $uaw$-null by \cite[Lemma 2]{Z}.
\end{remark}
Now, we have the following useful statement.
\begin{lemma}\label{2}
Every $AM$-space possesses {\bf DBSP}.
\end{lemma}
\begin{proof}
Suppose $E$ is an $AM$-space and $(x_n)\subseteq E$ is a norm bounded disjoint sequence in $E$. Then for each subsequence $(x_{n_k})$ from $(x_n)$, we have
\[\lim_{n\rightarrow \infty}\|\frac{1}{n}\Sigma_{k=1}^{n} x_{n_k}\|=\lim_{n\rightarrow \infty}\|\frac{1}{n}\bigvee_{k=1}^{n}x_{n_k}\|=\lim_{n\rightarrow \infty}\frac{1}{n}\bigvee_{k=1}^{n}\|x_{n_k}\|\leq \frac{1}{n}\rightarrow 0.\]
\end{proof}

For a Banach lattice $E$, by Lemma \ref{1}, {\bf UBSP} implies {\bf DBSP}. In the following, we show that, in order continuous Banach lattices, these notions agree.
\begin{theorem}\label{4}
A Banach lattice $E$ with {\bf DBSP} possesses {\bf UBSP} if and only if $E$ does not possess a lattice copy of $\ell_{\infty}$.

\end{theorem}
\begin{proof}
The "only if" part follows from that fact $\ell_{\infty}$ possesses {\bf DBSP} by Lemma \ref{2} but it does not have {\bf UBSP} since it fails to have {\bf WBSP}. Now, suppose $E$ is order continuous and $(x_n)$ is a bounded $uaw$-null sequence in $E$. By \cite[Theorem 4]{Z} and \cite[Theorem 3.2]{DOT}, there are a subsequence $(x_{n_k})$ of $(x_n)$ and a disjoint sequence $(d_k)$ such that $\|x_{n_k}-d_k\|\rightarrow 0$. By passing to a subsequence, we may assume that $\lim_{m\rightarrow \infty}\frac{1}{m}\Sigma_{i=1}^{m}d_{i}\rightarrow 0$. Now, the result follows from the following identity.
\[\|\frac{1}{m}\Sigma_{i=1}^{m}x_{n_i}-\frac{1}{m}\Sigma_{i=1}^{m}d_i\|\leq \frac{1}{m}\Sigma_{i=1}^{m}\|x_{n_i}-d_i\|\rightarrow 0.\]

\end{proof}

\begin{theorem}\label{801}
A Banach lattice $E$ with {\bf WBSP} possesses {\bf UBSP} if and only if it contains no lattice copy of $\ell_1$.

\end{theorem}
\begin{proof}
The "only if " part follows from this fact that $\ell_1$ possesses {\bf WBSP} but it does not have {\bf UBSP}; consider the standard basis $(e_n)$ which is $uaw$-null by \cite[Lemma 2]{Z}. For the converse, it follows from \cite[Theorem 4.69]{AB} that $E'$ is order continuous so that by Lemma \ref{1}, {\bf WBSP} implies {\bf UBSP}.

\end{proof}
Now, we establish the following result regarding the relation between {\bf DWBSP} and {\bf UBSP}.
\begin{corollary}
A $\sigma$-order complete Banach lattice $E$ with {\bf DWBSP} possesses {\bf UBSP} if and only if it does not possess a lattice copy of either $\ell_1$ or $\ell_{\infty}$.
\end{corollary}
\begin{proof}
The "only if" part is done by using this fact that both $\ell_1$ and $\ell_{\infty}$ possess {\bf DWBSP} but they fail to have {\bf UBSP}.
Now, suppose $E$ contains no lattice copy of either $\ell_1$ or $\ell_{\infty}$. By \cite[Theorem 4.69]{AB} and \cite[Theorem 4.56]{AB}, $E$ and $E'$ have order continuous norms. By \cite[Proposition 6.15]{GTX}, {\bf DWBSP} implies {\bf DBSP} and by Theorem \ref{4}, {\bf DBSP} results {\bf UBSP}.

\end{proof}
Suppose $E$ is an atomic $KB$-space. Recall that by \cite[Proposition 14]{Z}, every bounded sequence in $E$ has a $uaw$-convergent subsequence. So, we have the following.

\begin{proposition}
Suppose $E$ is an atomic $KB$-space which possesses {\bf UBSP}. Then it also has {\bf BSP}.

\end{proposition}

\end{document}